\documentclass[11pt]{amsart}

\usepackage[a4paper,hmargin=3.5cm,vmargin=4cm]{geometry}
\usepackage{amsfonts,amssymb,amscd,amstext}
\usepackage{graphicx}
\usepackage[dvips]{epsfig}

\usepackage{fancyhdr}
\input xy
\xyoption{all}

\usepackage{times}

\usepackage{enumerate}
\usepackage{titlesec}
\usepackage{mathrsfs}

\titleformat{\section}
{\filcenter\bfseries\large} {\thesection{.}}{0.2cm}{}
\titleformat{\subsection}[runin]
{\bfseries} {\thesubsection{.}}{0.15cm}{}[.]
\titleformat{\subsubsection}[runin]
{\em}{\thesubsubsection{.}}{0.15cm}{}[.]

\usepackage[up,bf]{caption}

\newtheorem{theorem}{Theorem}[section]

\newtheorem{lemma}[theorem]{Lemma}
\newtheorem{corollary}[theorem]{Corollary}

\theoremstyle{definition}

\newtheorem{problem}[theorem]{Problem}

\numberwithin{equation}{section}
\numberwithin{figure}{section}

\newcommand\Ascr{\mathscr{A}}

\newcommand\Cscr{\mathscr{C}}

\newcommand\Tscr{\mathscr{T}}

\renewcommand\b{\mathbb{B}}
\renewcommand\c{\mathbb{C}}

\renewcommand\d{\mathbb D}

\newcommand\n{\mathbb{N}}
\renewcommand\r{\mathbb{R}}

\newcommand\z{\mathbb{Z}}

\newcommand\sfb{\mathsf{b}}

\newcommand\igot{\mathfrak{i}}

\renewcommand\igot{\mathfrak{i}}

\newcommand\Fgot{\mathfrak{F}}

\renewcommand\imath{\igot}

\newcommand\hra{\hookrightarrow}

\newcommand\wt{\widetilde}

\newcommand\dist{\mathrm{dist}}

\newcommand\length{\mathrm{length}}

\def\dist{\mathrm{dist}}

\def\length{\mathrm{length}}

\usepackage{color}

\begin{document}


\fancyhead[LO]{Complete embedded complex curves in the ball of $\c^2$ can have any topology}
\fancyhead[RE]{A.\ Alarc\'on and J.\ Globevnik}
\fancyhead[RO,LE]{\thepage}

\thispagestyle{empty}


\vspace*{1cm}
\begin{center}
{\bf\LARGE 
Complete embedded complex curves in the ball of $\c^2$ can have any topology
}

\vspace*{0.5cm}

%
%
{\large\bf Antonio Alarc\'on \; and \; Josip Globevnik}
\end{center}


%
%
\vspace*{1cm}

\begin{quote}
{\small
\noindent {\bf Abstract}\hspace*{0.1cm}
In this paper we prove that the unit ball $\b $ of $\c^2$ admits complete properly embedded complex curves of any given topological type. Moreover, we provide examples containing any given closed discrete subset of $\b $.

\vspace*{0.1cm}

\noindent{\bf Keywords}\hspace*{0.1cm} Complex curves, holomorphic embeddings, complete bounded submanifolds.

\vspace*{0.1cm}

\noindent{\bf Mathematics Subject Classification (2010)}\hspace*{0.1cm} 32H02, 32B15, 32C22.
}
\end{quote}


\section{Introduction} 
\label{sec:intro}

In 1977 Yang raised the question whether there exist complete bounded complex submanifolds of a complex Euclidean space $\c^N$ $(N>1)$; see \cite{Yang1977,Yang1977JDG}. Recall that an immersed $k$-dimensional complex submanifold $\psi\colon M^k\to\c^N$ $(1\le k<N)$ is said to be {\em complete} if the Riemannian metric induced on $M$ by the Euclidean metric in $\c^N$ via $\psi$ is complete in the classical sense; equivalently, if the image by $\psi$ of any divergent path on $M$ has infinite Euclidean length. 

The case of main interest to us in this paper is when $k=1$ and $N=2$; i.e., complex curves in the complex Euclidean plane $\c^2$. There is a lot of known examples of complete bounded {\em immersed} complex curves in $\c^2$ 
which have been obtained by different methods; see Jones \cite{Jones1979PAMS} for discs, Mart\'in, Umehara, and Yamada \cite{MartinUmeharaYamada2009PAMS} for some finite topologies, Alarc\'on and L\'opez \cite{AlarconLopez2013MA} for examples with arbitrary topology, and Alarc\'on and Forstneri\v c \cite{AlarconForstneric2013MA} and Alarc\'on, Drinovec Drnov\v sek, Forstneri\v c, and L\'opez \cite{AlarconDrinovecForstnericLopez2015PLMS} for examples normalized by any given bordered Riemann surface. 
Furthermore, the curves in \cite{AlarconLopez2013MA,AlarconForstneric2013MA,AlarconDrinovecForstnericLopez2015PLMS} can be chosen proper in any given convex domain of $\c^2$; in particular, in the open unit Euclidean ball, which throughout this paper will be denoted by $\b$. 

On the other hand, using the techniques developed in the cited sources and taking into account that the general position of complex curves in $\c^N$ is embedded for all $N\ge 3$, it is not very hard to construct complete bounded {\em embedded} complex curves in $\c^N$ for any such $N$ (see again \cite{AlarconForstneric2013MA,AlarconDrinovecForstnericLopez2015PLMS}). Also in this line but for submanifolds of higher dimension, Alarc\'on and Forstneri\v c provided in \cite{AlarconForstneric2013MA} examples of complete bounded embedded $k$-dimensional complex submanifolds of $\c^{3k}$ for any $k\in\n$, whereas Drinovec Drnov\v sek proved in \cite{Drinovec2015JMAA} that every bounded, strictly pseudoconvex, smoothly bounded domain of $\c^k$ admits a complete proper holomorphic embedding into the unit ball of $\c^N$ provided that the codimension $N-k$ is sufficiently large.

However, constructing complete bounded {\em embedded} complex curves in $\c^2$ (and, more generally, complete bounded embedded complex hypersurfaces of $\c^N$ for $N>1$) is a much more arduous task; the main reason why is that {\em self-intersections of complex curves in $\c^2$} 
{\em are stable under small deformations}. It is therefore not surprising that the first known examples of 
such curves were only found almost four decades after Yang posed his question; they were Alarc\'on and L\'opez who gave them in \cite{AlarconLopez2016JEMS}. Their method, which in fact furnishes complete properly embedded complex curves in any convex domain of $\c^2$, is rather involved and relies, among other things, on a subtle self-intersection removal procedure that does not allow to infer any information on the topological type of the examples
. A little later Globevnik \cite{Globevnik2015AM,Globevnik2016MA}, using a different technique, extended the results in \cite{AlarconLopez2016JEMS} by proving the existence of complete properly embedded complex hypersurfaces in any pseudoconvex domain $D$ of $\c^N$ for any $N>1$; this settles the {\em embedded Yang problem} in an optimal way in all dimensions. 
His examples are given as level sets of highly oscillating holomorphic functions $D\to\c$, and hence, again, no information on their topology 
 is provided. In light of the above, the following questions naturally appear (cf.\ \cite[Question 1.5]{AlarconLopez2016JEMS} and \cite[Questions 13.1 and 13.2]{Globevnik2015AM}):
\begin{problem}\label{pro:topology}
Is there any restriction on the topology of a complete bounded embedded complex hypersurface of $\c^N$? What if $N=2$? 
For instance, do there exist complete proper holomorphic embeddings of the unit disk $\d\subset\c$ into the unit ball $\b\subset\c^2$?
\end{problem}

The first approach to this problem was recently made in \cite{AlarconGlobevnikLopez2016Crelle} by Alarc\'on, Globevnik, and L\'opez, who, with a conceptually new method based on the use of holomorphic automorphisms of $\c^N$, constructed complete closed complex hypersurfaces in the unit ball of $\c^N$ for any $N>1$ with certain control on their topology. In particular, for $N=2$, they showed that
{\em the unit ball of $\c^2$ admits complete properly embedded complex curves with arbitrary {\em finite} topology} (see \cite[Corollary 1.2]{AlarconGlobevnikLopez2016Crelle}), thereby affirmatively answering the third question in Problem \ref{pro:topology}. Going further in this line, Globevnik proved in \cite{Globevnik2016JMAA} the existence of complete proper holomorphic embeddings $\d\hra\b$ whose image contains any given closed discrete subset of $\b$. This is reminiscent to an old result by Forstneri\v c, Globevnik, and Stens{\o}nes asserting that, given a pseudoconvex Runge domain $D\subset\c^N$ $(N>1)$ and a closed discrete subset $\Lambda\subset D$, there is a proper holomorphic embedding $\d\hra D$ whose image contains $\Lambda$ (see \cite{ForstnericGlobevnikStensones1996MA}); it is nevertheless true that these embeddings are not ensured to be complete in any case. 

The aim of this paper is to settle Problem \ref{pro:topology} for $N=2$ by proving the existence of complete properly embedded complex curves in $\b$ with {\em arbitrary topology} (possibly infinite). This completely solves the problem and, in particular, generalizes Alarc\'on, Globevnik, and L\'opez's existence result \cite[Corollary 1.2]{AlarconGlobevnikLopez2016Crelle} which only deals with finite topological types. Moreover, we provide examples of such curves which contain any given closed discrete subset of $\b$,  thereby extending the above mentioned hitting result by Globevnik (see \cite[Theorem 1.1]{Globevnik2016JMAA}).

The main theorem of this paper can be stated as follows.
%
%
\begin{theorem}\label{th:main-intro}
Let $\Lambda$ be a closed discrete subset of the unit ball $\b \subset\c^2$. On each open connected orientable smooth surface $M$ there exists a complex structure such that the open Riemann surface $M$ admits a complete proper holomorphic embedding $M\hra\b$ whose image contains $\Lambda$.
%
\end{theorem}

It is perhaps worth mentioning that, choosing any closed discrete subset $\Lambda\subset\b$ such that $\overline\Lambda\setminus\Lambda=\sfb \b=\{\zeta\in\c^2\colon |\zeta|=1\}$, Theorem \ref{th:main-intro} trivially implies the following
\begin{corollary}\label{co:intro}
The unit ball $\b\subset\c^2$ contains complete properly embedded complex curves with any given topology and whose limit set equals $\sfb\b$.
\end{corollary}
Although it is not explicitly stated there, Corollary \ref{co:intro} in the simply-connected case straightforwardly follows from the results by Globevnik in \cite{Globevnik2016JMAA}.

Our method of proof exploits some ideas from both \cite{AlarconGlobevnikLopez2016Crelle} and \cite{Globevnik2016JMAA} (in particular, our construction technique is based on the use of holomorphic automorphisms of $\c^2$), but also from \cite{AlarconLopez2013JGA} where Alarc\'on and L\'opez constructed properly embedded complex curves in $\c^2$ with arbitrary topology. The latter contributes to the so-called embedding problem for open Riemann surfaces in $\c^2$; a long-standing open question in Riemann Surface Theory asking whether every open Riemann surface properly embeds in $\c^2$ as a complex curve (cf.\ Bell and Narasimhan \cite[Conjecture 3.7, page 20]{BellNarasimhan1990book}; for recent advances and a history of this classical problem we refer to the works by Forstneri\v c and Wold \cite{ForstnericWold2009JMPA,ForstnericWold2013APDE} and the references therein). It is shown in \cite{ForstnericWold2009JMPA} that given a compact bordered Riemann surface $\overline M=M\cup \sfb M$ admitting a smooth embedding $f\colon\overline M\hra\c^2$ which is holomorphic in $M$, there is a proper holomorphic embedding $\wt f\colon M\hra\c^2$ which is as close as desired to $f$ uniformly on a given compact subset of $M$. (A {\em compact bordered Riemann surface} $\overline M$ is a compact Riemann surface with boundary $\emptyset\neq 	\sfb M\subset \overline M$ consisting of finitely many pairwise disjoint smooth Jordan curves; its interior $M=\overline M\setminus \sfb M$ is called a {\em bordered Riemann surface}.) This fact and the arguments in its proof were key in Alarc\'on and L\'opez's construction method in \cite{AlarconLopez2013JGA} which will be used in the proof of Theorem \ref{th:main-intro}.

We strongly expect that the new construction techniques developed in this paper may be adapted to prove the statement of Theorem \ref{th:main-intro} but replacing the ball $\b$ by any convex domain of $\c^2$. The following questions, concerning pseudoconvex domains, remain open and seem to be much more challenging.

\begin{problem}\label{pro:pseudoconvex}
Let $D\subset\c^2$ be a pseudoconvex Runge domain. Does there exist a complete proper holomorphic embedding $\d\hra D$? Given a closed discrete subset $\Lambda\subset D$, do there exist complete properly embedded complex curves in $D$ containing $\Lambda$? 
\end{problem}

As we have already mentioned, every bordered Riemann surface $M$ admits a complete proper holomorphic {\em immersion} $M\to\b$ (see \cite{AlarconForstneric2013MA}), and if in addition there is a smooth embedding $\overline M\to\c^2$, being holomorphic in $M$, then $M$ properly holomorphically embeds into $\c^2$ (see \cite{ForstnericWold2009JMPA}). It is however an open question, likely very difficult, whether every bordered Riemann surface admits a holomorphic embedding in $\c^2$ (even without asking to the embedding any global condition such as completeness or properness); see e.g. the introduction of \cite{ForstnericWold2009JMPA} or Section 8.9 in the monograph by Forstneri\v c \cite{Forstneric2011book} for more information. Thus, one is also led to ask:
\begin{problem}\label{pro:BRS}
Let $\overline M=M\cup \sfb M$ be a compact bordered Riemann surface and assume that there is a smooth embedding $\overline M\hra \c^2$ which is holomorphic in $M$. Does $M$ admit complete holomorphic embeddings $M\hra\c^2$ with bounded image? 
\end{problem}

We hope to return to these interesting questions in a future work.

%
%
\subsection*{Organization of the paper}

In Section \ref{sec:prelim} we set the notation that will be used throughout the paper and, with the aim of making it self-contained, state some already known results which will be used in the proof of Theorem \ref{th:main-intro}. In Section \ref{sec:MT} we prove an approximation result by properly embedded complex curves in the unit ball $\b\subset\c^2$ (see Theorem \ref{th:main}) from which Theorem \ref{th:main-intro} will be easily derived. 


\section{Preliminaries}\label{sec:prelim}

We denote by $|\cdot|$ and $\b=\{\zeta\in\c^2\colon |\zeta|<1\}$ the Euclidean norm and the unit Euclidean ball in $\c^2$. For a subset $C\subset\c^2$ we denote by $\overline C$, $\mathring C$, and $\sfb C=\overline C\setminus\mathring C$ the topological closure, interior, and frontier of $C$ in $\c^2$, respectively. Also, given a point $\xi\in\c^2$ and a number $r\in\r_+=[0,+\infty[$ we write $\xi+rC=\{\xi+r\zeta\colon \zeta\in C\}$.

Let $A$ be a smoothly bounded compact domain in an open Riemann surface and let $k\in\z_+=\{0,1,2,...\}$. We denote by $\Ascr^k(A)$ the space of functions $A\to\c$ of class $\Cscr^k(A)$ which are holomorphic on the interior $\mathring A=\overline A\setminus \sfb A$. If $N\in\n$ we will simply write $\Ascr^k(A)$ instead of $\Ascr^k(A)^N=\Ascr^k(A)\times \stackrel{\text{$N$ times}}{\cdots}\times \Ascr^k(A)$ when there is no place for ambiguity. Given an immersion $\psi\colon A\to\c^N$ $(N>1)$ of class $\Cscr^1(A)$, we denote by $\dist_\psi\colon A\times A\to\r_+$ the Riemannian distance in $A$ induced by the Euclidean metric of $\c^N$ via $\psi$; that is:
\[
	\dist_\psi(p,q):=\inf\{\length(\psi(\gamma))\colon \text{ $\gamma\subset A$ path connecting $p$ and $q$}\},\quad p,q\in A,
\]
where $\length(\cdot)$ denotes the Euclidean length in $\c^N$.

\subsection{Tangent balls}\label{ss:TB}

Given a point $\zeta\in\c^2\setminus\{0\}$ and a number $r>0$, we denote by  $\Tscr(\zeta,r)$ the closed ball with center $\zeta$ and radius $r$ in the real affine hyperplane $H_\zeta$ tangent to the sphere $\sfb(|\zeta|\b)$ at the point $\zeta$, that is:
\[
	\Tscr(\zeta,r):=\{\xi\in H_\zeta\colon |\xi-\zeta|\le r\}\subset\c^2.
\]
According to Alarc\'on, Globevnik, and L\'opez \cite[Definition 1.3]{AlarconGlobevnikLopez2016Crelle}, the set $\Tscr(\zeta,r)$ above is called {\em the tangent ball of center $\zeta$ and radius $r$}. A collection $\Fgot=\{\Tscr(\zeta_j,r_j)\}_{j\in J}$ of tangent balls in $\b\subset\c^2$ will be called {\em tidy} (cf.\ \cite[Definition 1.4]{AlarconGlobevnikLopez2016Crelle}) if it satisfies the following requirements:
\begin{enumerate}[A]
\item[$\bullet$] $\Tscr(\zeta_j,r_j)\subset\b$ for all $j\in J$ and the tangent balls in $\Fgot$ are pairwise disjoint.
\item[$\bullet$] $t\overline\b$ intersects finitely many tangent balls in the family $\Fgot$ for all $0<t<1$.
\item[$\bullet$] If $\Tscr(\zeta,r)$, $\Tscr(\zeta',r') \in\Fgot$ and $|\zeta|=|\zeta'|$, then $r=r'$.
\item[$\bullet$] If $\Tscr(\zeta,r)$, $\Tscr(\zeta',r') \in\Fgot$ and $|\zeta|<|\zeta'|$, then $\Tscr(\zeta,r)\subset|\zeta'|\b$.
\end{enumerate}
Notice that a tidy collection $\Fgot=\{\Tscr(\zeta_j,r_j)\}_{j\in J}$ of tangent balls in $\b$ consists of countably many elements; and so we may assume that $J\subset\n=\{1,2,3,...\}$. We denote by
\[
	|\Fgot|:=\bigcup_{j\in J} \Tscr(\zeta_j,r_j)
\]
the union of all the tangent balls in a tidy collection. Note that if $J$ is finite then $|\Fgot|$ is compact, whereas if $J$ is infinite then $|\Fgot|$ is a proper subset of $\b$.

Alarc\'on, Globevnik, and L\'opez proved in \cite{AlarconGlobevnikLopez2016Crelle} the following two results, involving tidy collections of tangent balls, which will be invoked in our argumentation.
%
%
\begin{lemma}[\text{\cite[Lemma 2.4]{AlarconGlobevnikLopez2016Crelle}}]\label{lem:building}
Given numbers $0<r<r'<1$ and $\ell>0$ there is a finite tidy collection $\Fgot$ of tangent balls in $\b$ such that $|\Fgot|\subset r'\b\setminus r\overline \b$ and the length of any path $\gamma\colon [0,1]\to r'\overline\b\setminus ((r\b)\cup |\Fgot|)$ with $|\gamma(0)|=r$ and $|\gamma(1)|=r'$ is at least $\ell$.
\end{lemma}

The next result is not explicitly stated in \cite{AlarconGlobevnikLopez2016Crelle} but straightforwardly follows from an inspection of the proof of \cite[Theorem 1.6]{AlarconGlobevnikLopez2016Crelle} as a standard recursive application of \cite[Lemma 3.1]{AlarconGlobevnikLopez2016Crelle}.
%
%
\begin{lemma}\label{lem:avoiding} Let $0<r<r'<1$ be numbers and let $\Fgot$ be a finite tidy collection of tangent balls in $\b$ with $|\Fgot|\subset r'\b\setminus r\overline \b$. Then, given a properly embedded complex curve $Z\subset\c^2$ and a number $\epsilon>0$, there is a holomorphic automorphism $\Phi\colon\c^2\to\c^2$ satisfying the following properties:
\begin{enumerate}[\rm (i)]
\item $\Phi(Z)\cap |\Fgot|=\emptyset$.
\item $|\Phi(\zeta)-\zeta|<\epsilon$ for all $\zeta\in r\overline\b$.
\end{enumerate}
\end{lemma}

\subsection{Hitting and approximation lemmas}\label{ss:hitting}


In this subsection we state two results which will also be used in our construction.
The first one, due to Globevnik, will be key in order to achieve the hitting condition in the statement of Theorem \ref{th:main-intro}.
%
%
\begin{lemma}[\text{\cite[Lemma 7.2]{Globevnik2016JMAA}}]\label{lem:hiting}
Given a finite subset $\Lambda\subset\b$ there exist numbers $\eta>0$ and $\mu>0$ such that the following holds. Given a number $0<\delta<\eta$ and a 
%
%
map $\varphi\colon\Lambda\to\c^2$ such that
\[
	|\varphi(\zeta)-\zeta|<\delta\quad \text{for all $\zeta\in\Lambda$},
\]
there exists a holomorphic automorphism $\Psi\colon\c^2\to\c^2$ satisfying the following conditions:
\begin{enumerate}[\rm (i)]
\item $\Psi(\varphi(\zeta))=\zeta$ for all $\zeta\in\Lambda$.
\item $|\Psi(\zeta)-\zeta|<\mu\delta$ for all $\zeta\in\overline\b$.
\end{enumerate}
\end{lemma}

%
%
The second one, which will help us to increase the topology in the recursive construction, is a particular case of a result by Alarc\'on and L\'opez \cite[Theorem 4.5]{AlarconLopez2013JGA}; it also easily follows from the results by Forstneri\v c and Wold in \cite{ForstnericWold2009JMPA}.
\begin{lemma}\label{lem:AL}
Let $\overline M=M\cup \sfb M$ be a compact bordered Riemann surface and let $K\subset M$ be a connected, smoothly bounded, Runge compact domain which is a strong deformation retract of $\overline M$. Let $\phi\colon K\hra\c^2$ be an embedding of class $\Ascr^1(K)$ and assume that there exists a number $s>0$ such that
\[
	\phi(bK)\cap s\overline\b=\emptyset.
\]
Then, given $\epsilon>0$, there are an open domain $\Omega\subset M$ and a proper holomorphic embedding $\wt\phi\colon \Omega\hra\c^2$ such that:
\begin{enumerate}[\rm (i)]
\item $K\subset \Omega$ and $\Omega$ is a deformation retract of $M$ and homeomorphic to $M$.
\item $|\wt\phi(p)-\phi(p)|<\epsilon$ for all $p\in K$.
\item $\wt\phi(\Omega\setminus\mathring K)\cap s\overline\b=\emptyset$.
\end{enumerate}
\end{lemma}
Recall that a compact subset $K$ of an open Riemann surface $M$ is said to be {\em Runge} or {\em holomorphically convex} if every continuous function $K\to\c$, holomorphic in $\mathring K$, may be approximated uniformly on $K$ by holomorphic functions $M\to\c$. The classical Mergelyan theorem ensures that $K\subset M$ is Runge if and only if $M\setminus K$ has no relatively compact connected components in $M$ (see Bishop \cite{Bishop1958PJM} and also Runge \cite{Runge1885AM} and Mergelyan \cite{Mergelyan1951DAN}).


\section{Proof of the main theorem}\label{sec:MT}

In this section we prove the following more precise version of Theorem \ref{th:main-intro}.

\begin{theorem}\label{th:main}
Let $M$ be an open connected Riemann surface, let $K\subset M$ be a connected, smoothly bounded, Runge compact domain, let $0<s<r<1$ be numbers, and assume that there is an embedding $\psi\colon K\to\c^2$ of class $\Ascr^1(K)$ such that 
\begin{equation}\label{eq:th-rs}
	\psi(\sfb K)\subset r\b\setminus s\overline\b.
\end{equation}
Let $\Lambda\subset \b $ be a closed discrete subset such that 
\begin{equation}\label{eq:th-Lambda}
	\Lambda\cap r\overline\b \subset\psi(\mathring K)\cap s\b.
\end{equation}
Then, given $\epsilon>0$, there are a domain $D\subset M$ and a complete proper holomorphic embedding $\wt\psi\colon D\hra\b $ satisfying the following properties:
\begin{enumerate}[\rm (i)]
\item $K\subset D$ and $D$ is a deformation retract of $M$ and homeomorphic to $M$.
\item $|\wt \psi(\zeta)-\psi(\zeta)|<\epsilon$ for all $\zeta\in K$.
\item $\Lambda\subset \wt \psi(D)$.
\item $\wt\psi(D\setminus K)\cap s\overline\b=\emptyset$.
\end{enumerate}
\end{theorem}
\begin{proof}
Pick a number $s<s_0'<r$ such that
\begin{equation}\label{eq:r1s0'}
	\psi(\sfb K)\subset r\b\setminus s_0'\overline\b.
\end{equation}
Such an $s_0'$ exists in view of \eqref{eq:th-rs} by compactness of $\psi(\sfb K)$. Without loss of generality we assume that $\Lambda$ is infinite, and hence, since it is closed in $\b$ and discrete, $\Lambda$ is in bijection with $\n$ and for any ordering $\Lambda=\{p_i\}_{i\in\n}$ of $\Lambda$ we have that $\lim_{i\to\infty}|p_i|=1$. Thus, there are sequences of numbers $\{r_j\}_{j\in\n}$, with $r_1=r$, $\{r_j'\}_{j\in\n}$, $\{s_j\}_{j\in\n}$, and $\{s_j'\}_{j\in\n}$, satisfying the following properties:
\begin{enumerate}[\rm (a)]
\item $s_{j-1}'<r_j<r_j'<s_j<s_j'$ for all $j\in\n$.
\item $\lim_{j\to\infty} r_j=1$.
\item $\Lambda\subset s\b\cup \big(\bigcup_{j\in\n} (s_j'\b\setminus s_j\overline \b)\big)$. Take into account \eqref{eq:th-Lambda}.
\end{enumerate}
Write $\Lambda_0:=\Lambda\cap s\b=\Lambda\cap s_0'\b$ and $\Lambda_j=\Lambda\cap (s_j'\b\setminus s_j\overline \b)$, $j\in\n$, and observe that $\Lambda_j$ is finite for all $j\in\z_+$. We also assume without loss of generality that $\Lambda_j\neq\emptyset$ for all $j\in\z_+$. By {\rm (a)} and {\rm (c)} we have that 
\begin{equation}\label{eq:Lambdaj}
	\Lambda=\bigcup_{j\in\z_+}\Lambda_j,\quad \Lambda_i\cap\Lambda_j=\emptyset\text{ for all $i, j\in\z_+$, $i\neq j$}.
\end{equation}

Let $\{\epsilon_j\}_{j\in\n}\searrow 0$ be a decreasing sequence of positive numbers such that 
\begin{equation}\label{eq:<epsilon}
	\sum_{j\in\n}\epsilon_j<\epsilon.
\end{equation}
The precise values of the numbers $\epsilon_j$, $j\in\n$, will be specified later.

Write $M_0:=K$ and let
\begin{equation}\label{eq:Mj}
	M_0\Subset M_1\Subset M_2\Subset\ldots\Subset M=\bigcup_{j\in\z_+} M_j
\end{equation}
be an exhaustion of $M$ by connected, smoothly bounded, Runge compact domains such that the Euler characteristic $\chi(M_j\setminus\mathring M_{j-1})\in\{0,-1\}$ for all $j\in\n$. Such exists by basic topological arguments (see e.g.\ \cite[Lemma 4.2]{AlarconLopez2013JGA} for a simple proof).

Write $D_0:=M_0=K$, denote by $\iota_0\colon D_0\to M_0$ the identity map, and set $\psi_0:=\psi$. Fix a point $p_0\in\mathring K$. We shall recursively construct a sequence $\Xi_j=\{D_j,\psi_j\}$, $j\in\n$, where $D_j$ is a connected, smoothly bounded, Runge compact domain in $M$ and $\psi_j\colon D_j\hra \c^2$ is an embedding of class $\Ascr^1(D_j)$ such that the following hold for all $j\in\n$.
\begin{enumerate}[\rm (1$_j$)]
\item $D_{j-1}\Subset D_j\Subset M_j$ and $D_j$ is a strong deformation retract of $M_j$.
\item $|\psi_j(p)-\psi_{j-1}(p)|<\epsilon_j$ for all $p\in D_{j-1}$.
\item $\dist_{\psi_j}(p_0,\sfb D_j)>j$.
\item $\psi_j(\sfb D_j)\subset r_{j+1}\b\setminus s_j'\overline\b$.
\item $\psi_j(D_j\setminus \mathring D_{j-1})\cap s_{j-1}'\overline\b=\emptyset$.
\item $\Lambda_i\subset\psi_j(\mathring D_i)$ for all $i\in\{0,\ldots,j\}$.
\end{enumerate}

Assume for a moment that we have already constructed a sequence $\{\Xi_j\}_{j\in\n}$ enjoying the above conditions. By {\rm (1$_j$)} and \eqref{eq:Mj}, we obtain that
\[
	D:=\bigcup_{j\in\z_+} D_j
\]
is a domain in $M$ satisfying property {\rm (i)} in the statement of the theorem. We claim that, if the number $\epsilon_j>0$ is chosen small enough at each step in the recursive construction, the sequence $\{\psi_j\}_{j\in\z_+}$ converges uniformly on compacta in $D$ to a limit map 
\[
	\wt\psi=\lim_{j\to\infty}\psi_j\colon D\to\c^2
\]
satisfying the conclusion of the theorem. Indeed, by {\rm (2$_j$)} and \eqref{eq:<epsilon} the limit map $\wt \psi$ exists and satisfies {\rm (ii)}; recall that $K=D_0$.  Since each map $\psi_j$ is an embedding of class $\Ascr^1(D_j)$, choosing the number $\epsilon_j>0$ sufficiently small at each step in the recursive construction, then $\wt\psi\colon D\to\c^2$ is an injective holomorphic immersion. Moreover, {\rm (3$_j$)} ensures that $\wt \psi$ is complete, whereas {\rm (5$_j$)}, {\rm (a)}, and {\rm (b)} guarantee {\rm (iv)} and the facts that $\wt\psi(D)\subset\b$ and that $\wt\psi\colon D\to\b$ is a proper map; recall that $s<s_0'$ and observe that $\lim_{j\to\infty} s_j'=1$. Thus, $\wt\psi\colon D\hra\b$ is a proper embedding. Finally, {\rm (6$_j$)} and \eqref{eq:Lambdaj} imply that $\Lambda\subset\wt\psi(D)$, thereby proving {\rm (iii)}.

To conclude the proof it therefore suffices to construct a sequence $\Xi_j=\{D_j,\psi_j\}$, $j\in\n$, satisfying properties {\rm (1$_j$)}--{\rm (6$_j$)} above. We proceed by induction. The basis is given by the compact domain $D_0$ and the map $\psi_0$; observe that {\rm (3$_0$)} holds since $p_0\in \mathring D_0$ and $\psi_0$ is an immersion, {\rm (4$_0$)}$=$\eqref{eq:r1s0'} (recall that $r_1=r$), {\rm (6$_0$)} is implied by {\rm (c)} and \eqref{eq:th-Lambda}, and {\rm (2$_0$)}, {\rm (5$_0$)}, and the first part of {\rm (1$_0$)},  are vacuous conditions. Finally, the second part of {\rm (1$_0$)} is obvious since $D_0=M_0$.

For the inductive step, assume that we have $\Xi_{j-1}=\{D_{j-1},\psi_{j-1}\}$ enjoying the desired properties for some $j\in\n$ and let us construct $\Xi_j=\{D_j,\psi_j\}$. We distinguish cases.

\smallskip

\noindent{\em Case 1: Assume that the Euler characteristic $\chi(M_j\setminus\mathring M_{j-1})=0$.} In this case $M_{j-1}$ is a strong deformation retract of $M_j$. 

Write $\Lambda'=\bigcup_{i=0}^j\Lambda_i$ and let $\eta>0$ and $\mu>0$ be the numbers given by Lemma \ref{lem:hiting} applied to the finite subset $\Lambda'\subset\b$. 
By {\rm (3$_{j-1}$)} and {\rm (4$_{j-1}$)} there is another number $\eta'>0$ with the following property:
\begin{enumerate}[\rm ({A}1)]
\item If $\phi\colon D_{j-1}\to\c^2$ is an immersion of class $\Ascr^1(D_{j-1})$ such that $|\phi(p)-\psi_{j-1}(p)|<\eta'$ for all $p\in D_{j-1}$, then $\dist_\phi(p_0,\sfb D_{j-1})>j-1$ and $\phi(\sfb D_{j-1})\subset r_j\b\setminus s_{j-1}'\overline\b$.
\end{enumerate}

On the other hand, Lemma \ref{lem:building} gives a finite tidy collection $\Fgot$ of tangent balls in $\b$ satisfying the following conditions:
\begin{enumerate}[\rm ({A}1)]
\item[\rm ({A}2)] $|\Fgot|\subset r_j'\b\setminus r_j\overline\b$.
\item[\rm ({A}3)] The length of any path $\gamma\colon [0,1]\to r_j'\overline\b\setminus ((r_j\b)\cup |\Fgot|)$ with $|\gamma(0)|=r_j$ and $|\gamma(1)|=r_j'$ is greater than $2$.
\end{enumerate}
Thus, there is a third number $\eta''>0$ enjoying the following property:
\begin{enumerate}[\rm ({A}1)]
\item[\rm ({A}4)] If $\alpha\colon [0,1]\to r_j'\overline\b\setminus r_j\b$ is a path satisfying that there is another path $\gamma\colon [0,1]\to r_j'\overline\b\setminus ((r_j\b)\cup |\Fgot|)$ such that $|\gamma(0)|=r_j$, $|\gamma(1)|=r_j'$, and $|\gamma(x)-\alpha(x)|<\eta''$ for all $x\in[0,1]$, then the length of $\alpha$ is greater than $1$.
\end{enumerate}

Next, pick a number $t$ such that $s_{j-1}'<t<r_j$ and
\begin{equation}\label{eq:t}
	\psi_{j-1}(\sfb D_{j-1})\subset r_j\b\setminus t\overline\b;
\end{equation}
existence of such number $t$ is ensured by {\rm (a)} and {\rm (4$_{j-1}$)}.
Finally, choose a number
\begin{equation}\label{eq:delta<eta}
	0<\delta<\min\Big\{\eta ,\frac{\eta'}{\mu+1},\frac{\eta''}{\mu+1},\frac{\epsilon_j}{\mu+1},\frac{t-s_{j-1}'}{\mu},\frac{r_{j+1}-s_j'}{2\mu}\Big\}.
\end{equation}

Also fix a number $\tau_1>0$ which will be specified later. 

Taking into account \eqref{eq:t}, Lemma \ref{lem:AL} furnishes an open domain $\Omega\Subset \mathring M_j$ and a proper holomorphic embedding $\phi_1\colon \Omega\hra\c^2$ such that the following hold.
\begin{enumerate}[\rm ({B}1)]
\item $D_{j-1}\subset \Omega$ and $\Omega$ is a deformation retract of $\mathring M_j$ and homeomorphic to $\mathring M_j$. In particular, the second part of {\rm (1$_{j-1}$)} and the fact that $M_{j-1}$ is a strong deformation retract of $M_j$ ensure that $\Omega\setminus D_{j-1}$ consists of a finite collection of pairwise disjoint open annuli.
\item $|\phi_1(p)-\psi_{j-1}(p)|<\tau_1$ for all $p\in D_{j-1}$.
\item $\phi_1(\Omega\setminus \mathring D_{j-1})\cap t\overline\b=\emptyset$.
\end{enumerate}
In view of \eqref{eq:t} and {\rm (B2)}, and choosing $\tau_1>0$ small enough, we also have that
\begin{enumerate}[\rm ({B}1)]
\item [\rm ({B}4)] $\phi_1(\sfb D_{j-1})\subset r_j\b\setminus t\overline\b$.
\end{enumerate} 

Now, given a number $\tau_2>0$ which will be specified later and taking into account {\rm (A2)}, Lemma \ref{lem:avoiding} provides a holomorphic automorphism $\Phi\colon\c^2\to\c^2$ such that:
\begin{enumerate}[\rm ({C}1)]
\item $\Phi(\phi_1(\Omega))\cap |\Fgot|=\emptyset$.
\item $|\Phi(\zeta)-\zeta|<\tau_2$ for all $\zeta\in r_j\overline\b$.
\end{enumerate}
Write 
\begin{equation}\label{eq:phi2}
\phi_2:=\Phi\circ\phi_1\colon\Omega\hra\c^2.
\end{equation}
By {\rm (B1)}, {\rm (B3)}, {\rm (B4)}, {\rm (C2)}, and {\rm (a)}, and choosing $\tau_2>0$ sufficiently small, we have that
\begin{enumerate}[\rm ({D}1)]
\item $\phi_2(\sfb D_{j-1})\subset r_j\b\setminus t\overline\b$,
\item $\phi_2(\Omega\setminus \mathring D_{j-1})\cap t\overline\b=\emptyset$,
\end{enumerate}
and, taking into account {\rm (B1)}, {\rm (D1)}, and the Maximum Principle for the function $|\phi_2|$, there exists a smoothly bounded compact domain $\Upsilon \subset\Omega$ such that: 
\begin{enumerate}[\rm ({D}1)]
\item[\rm ({D}3)] $D_{j-1}\Subset \Upsilon $ and $D_{j-1}$ is a strong deformation retract of $\Upsilon $.
\item[\rm ({D}4)] $\phi_2(\sfb \Upsilon )\subset \sfb (t'\b)$ and meets transversely there for some number $t'$ with $r_j'<t'<s_j$.
\end{enumerate}

Next, choose a point $q_0\in \sfb \Upsilon $ and a smooth embedded compact path $\lambda\subset \c^2\setminus t'  \b$ having $\phi_2(q_0)$ as an endpoint, meeting $\sfb (t'  \b)$ transversely there, and being otherwise disjoint from $t'  \overline\b$. Assume also that
\begin{equation}\label{eq:hitting1}
	\Lambda_j\subset \lambda\subset s_j'\b\setminus t'  \b.
\end{equation} 
Also take a smooth embedded compact path $\gamma\subset\Omega\setminus\mathring\Upsilon $ having $q_0$ as an endpoint, meeting $\sfb \Upsilon $ transversely there, and being otherwise disjoint from $\Upsilon $. Extend $\phi_2$, with the same name, to a smooth embedding $\Upsilon \cup\gamma\hra\c^2$ such that 
\begin{equation}\label{eq:phi2lambda}
	\phi_2(\gamma)=\lambda.
\end{equation}
Observe that $\Upsilon\cup\gamma$ is a Runge compact subset of $\Omega$; take into account {\rm (B1)} and {\rm (D3)}.
Thus, given $\tau_3>0$ to be specified later, Mergelyan's theorem applied to $\phi_2\colon \Upsilon \cup\gamma\hra\c^2$ ensures the existence of a smoothly bounded compact domain $\Upsilon' \subset\Omega$ and an embedding $\phi_3\colon \Upsilon' \hra\c^2$ of class $\Ascr^1(\Upsilon' )$ such that:
\begin{enumerate}[\rm ({E}1)]
\item $\Upsilon \cup\gamma\subset\mathring\Upsilon' $ and $\Upsilon' $ is a strong deformation retract of $M_j$. In particular, $D_{j-1}$ is a strong deformation retract of $\Upsilon'$; see {\rm (D3)}
\item $|\phi_3(p)-\phi_2(p)|<\tau_3$ for all $p\in\Upsilon \cup\gamma$.
\end{enumerate}
Furthermore, if we take $\Upsilon' $ close enough to $\Upsilon \cup\gamma$ and if $\tau_3>0$ is chosen sufficiently small, we obtain in view of {\rm (E2)} that:
\begin{enumerate}[\rm ({E}1)]
\item[\rm (E3)] $\phi_3(\Upsilon' \setminus\mathring\Upsilon )\subset s_j'\b\setminus r_j'\overline\b$. See {\rm (D4)}, {\rm (a)}, \eqref{eq:hitting1}, and \eqref{eq:phi2lambda}.
\item[\rm (E4)] $\phi_3(\sfb D_{j-1})\subset r_j\b\setminus t\overline\b$. See {\rm (D1)} and {\rm (D3)}.
\item[\rm (E5)] $\phi_3(\Upsilon' \setminus\mathring D_{j-1})\cap t\overline\b=\emptyset$. See {\rm (D2)}.
\item[\rm (E6)] $\phi_3(\Upsilon' )\cap|\Fgot|=\emptyset$. See \eqref{eq:phi2}, {\rm (C1)}, {\rm (E3)}, and {\rm (A2)}.
\end{enumerate}

Given $\tau_4>0$ to be specified later, applying Lemma \ref{lem:AL} once again, we obtain, in view of {\rm (E3)}, a domain $\Omega'\Subset M_j$ and a proper holomorphic embedding $\phi_4\colon \Omega'\hra\c^2$ such that:
\begin{enumerate}[\rm ({F}1)]
\item $\Upsilon' \subset\Omega'$ and $\Omega'$ is a deformation retract of $\mathring M_j$ and homeomorphic to $\mathring M_j$. In particular, $\Omega'\setminus \Upsilon' $ consists of a finite collection of pairwise disjoint open annuli. (Cf.\ {\rm (B1)}.)
\item $|\phi_4(p)-\phi_3(p)|<\tau_4$ for all $p\in\Upsilon' $.
\item $\phi_4(\Omega'\setminus\mathring \Upsilon' )\cap r_j'\overline\b=\emptyset$.
\end{enumerate}
If $\tau_4>0$ is chosen sufficiently small then, in view of {\rm (F2)}, we also have:
\begin{enumerate}[\rm ({F}1)]
\item[\rm (F4)] $\phi_4(\sfb D_{j-1})\subset r_j\b\setminus t\overline\b$. See {\rm (E4)}.
\item[\rm (F5)] $\phi_4(\Omega'\setminus\mathring D_{j-1})\cap t\overline\b=\emptyset$. See {\rm (F3)} and {\rm (E5)}, and recall that $t<r_j<r_j'$.
\item[\rm (F6)] $\phi_4(\Omega')\cap|\Fgot|=\emptyset$. See {\rm (E6)}, {\rm (F3)}, and {\rm (A2)}.
\end{enumerate}

Assume now that the numbers $\tau_i$, $i=1,\ldots,4$, are chosen small enough so that $\sum_{i=1}^4\tau_i<\delta$, where $\delta>0$ is the number in \eqref{eq:delta<eta}. Thus, by {\rm (B2)}, {\rm (C2)}, \eqref{eq:phi2}, {\rm (E2)}, and {\rm (F2)}, we have:
\begin{enumerate}[\rm ({F}1)]
\item[\rm (F7)] $|\phi_4(p)-\psi_{j-1}(p)|<\sum_{i=1}^4\tau_i<\delta$ for all $p\in D_{j-1}$.
\item[\rm (F8)] $|\phi_4(p)-\phi_2(p)|<\tau_3+\tau_4<\sum_{i=1}^4\tau_i<\delta$ for all $p\in \Upsilon \cup\gamma$.
\end{enumerate}

By {\rm (6$_{j-1}$)}, {\rm (c)}, \eqref{eq:hitting1}, and \eqref{eq:phi2lambda}, we infer that $\Lambda_i\subset \psi_{j-1}(\mathring D_i)$, $i=0,\ldots,j-1$, and $\Lambda_j\subset\phi_2(\gamma)$. Thus, {\rm (F7)} and {\rm (F8)} guarantee the existence of an injective map $\varphi\colon\Lambda'\to \phi_4(\Omega')\subset\c^2$ such that:
\begin{enumerate}[\rm ({G}1)]
\item	$|\varphi(\zeta)-\zeta|<\delta$ for all $\zeta\in\Lambda'$.
\item $\varphi(\Lambda_i)\subset \phi_4(\mathring D_i)$ for all $i\in\{0,\ldots,j-1\}$.
\item $\varphi(\Lambda_j)\subset \phi_4(\gamma)\subset \phi_4(\mathring \Upsilon' \setminus \Upsilon )$.
\end{enumerate}
(Observe that, since $0<\delta<\eta$ where $\eta$ is given by Lemma \ref{lem:hiting} for the subset $\Lambda'\subset\b$ (see \eqref{eq:delta<eta}), every map $\varphi\colon \Lambda'\to\c^2$ satisfying {\rm (G1)} is injective.) In view of \eqref{eq:delta<eta} and {\rm (G1)}, Lemma \ref{lem:hiting} provides a holomorphic automorphism $\Psi\colon\c^2\to\c^2$ such that:
\begin{enumerate}[\rm ({H}1)]
\item $\Psi(\varphi(\zeta))=\zeta$ for all $\zeta\in\Lambda'$.
\item $|\Psi(\zeta)-\zeta|<\mu\delta$ for all $\zeta\in\overline\b$, where $\mu$ is the number given by Lemma \ref{lem:hiting} for the set $\Lambda'$ which appears in \eqref{eq:delta<eta}. 
\end{enumerate}
Consider the proper holomorphic embedding 
\[
	\phi:=\Psi\circ\phi_4\colon \Omega'\hra\c^2.
\]
Properties {\rm (F4)}, {\rm (F7)}, {\rm (H2)}, and \eqref{eq:delta<eta}, together with the Maximum Principle, ensure that
\begin{equation}\label{eq:mu+1}
	|\phi(p)-\psi_{j-1}(p)|<(\mu+1)\delta<\min\{\eta',\eta'',\epsilon_j\}\quad 
	\text{for all $p\in D_{j-1}$}.
\end{equation}
This inequality and {\rm (A1)} guarantee that
\begin{equation}\label{eq:dist}
\dist_\phi(p_0,\sfb D_{j-1})>j-1
\end{equation}
and
\begin{equation}\label{eq:bDj-1}
	\phi(\sfb D_{j-1})\subset r_j\b\setminus s_{j-1}'\overline\b.
\end{equation}
Furthermore, since $\phi_4\colon\Omega'\hra\c^2$ is a proper holomorphic embedding, {\rm (F4)}, {\rm (F1)}, and the Maximum Principle ensure the existence of a smoothly bounded compact domain $D_j\subset \mathring M_j$ satisfying condition {\rm (1$_j$)} and such that
\begin{equation}\label{eq:phi4bDj}
\phi_4(\sfb D_j)\subset (r_{j+1}-\mu\delta)\b\setminus (s_j'+\mu\delta)\overline\b \subset\overline\b;
\end{equation}
observe that $r_{j+1}-\mu\delta>s_j'+\mu\delta>r_j$ by \eqref{eq:delta<eta} and {\rm (a)}. Thus, {\rm (H2)} gives that
\begin{equation}\label{eq:bDj}
\phi(\sfb D_j)\subset r_{j+1}\b\setminus s_j'\overline\b.
\end{equation}
Moreover, \eqref{eq:phi4bDj} and the Maximum Principle imply that $\phi_4(D_j)\subset (r_{j+1}-\mu\delta)\b\subset\overline\b$, and hence {\rm (H2)} gives that
\begin{equation}\label{eq:mu+1'}
	|\phi(p)-\phi_4(p)|<\mu\delta
	\quad \text{for all $p\in D_j$}.
\end{equation}

Set
\[
	\psi_j:=\phi|_{D_j}\colon D_j\to\c^2.
\]
We claim that the pair $\Xi_j=\{D_j,\psi_j\}$ satisfies conditions {\rm (1$_j$)}--{\rm (6$_j$)}.  
Indeed, {\rm (1$_j$)} has been already checked, {\rm (2$_j$)} is implied by \eqref{eq:mu+1}, and {\rm (4$_j$)}=\eqref{eq:bDj}. On the other hand, since $\phi_4(D_j)\subset\b$, properties {\rm (F5)} and {\rm (H2)} ensure that $\emptyset=\psi_j(D_j\setminus\mathring D_{j-1})\cap (t-\mu\delta)\overline\b\supset \psi_j(D_j\setminus\mathring D_{j-1})\cap s_{j-1}'\overline\b$, where the latter inclusion follows from the fact that $s_{j-1}'<t-\mu\delta$ (see \eqref{eq:delta<eta}). This proves {\rm (5$_j$)}. Further, {\rm (G2)} and {\rm (H1)}  guarantee that $\Lambda_i\subset\psi_j(\mathring D_i)$ for all $i\in\{0,\ldots,j-1\}$, whereas {\rm (G3)}, {\rm (H1)}, and \eqref{eq:bDj} ensure that $\Lambda_j\subset\psi_j(\mathring D_j)$; thereby proving {\rm (6$_j$)}. Finally, \eqref{eq:bDj-1}, \eqref{eq:bDj}, \eqref{eq:mu+1'}, {\rm (F6)}, and {\rm (A4)} guarantee that $\dist_{\psi_j}(\sfb D_{j-1},\sfb D_j)>1$; take into account that $s_j'>r_j'$ in view of {\rm (a)}. This and \eqref{eq:dist} prove {\rm (3$_j$)}; recall that $p_0\in\mathring D_{j-1}\Subset D_j$.

This concludes the construction of $\Xi_j$ in Case 1.

\smallskip

\noindent{\em Case 2: Assume that the Euler characteristic $\chi(M_j\setminus \mathring M_{j-1})=-1$.} In this case there is a smooth embedded compact path $\gamma\subset\mathring M_j\setminus\mathring D_{j-1}$ with its two endpoints lying in $\sfb D_{j-1}$, meeting transversely there, and otherwise disjoint from $D_{j-1}$, such that $D_{j-1}\cup\gamma$ is a strong deformation retract of $M_j$. 

By {\rm (4$_{j-1}$)} we may extend $\psi_{j-1}$, with the same name, to a smooth embedding $D_{j-1}\cup\gamma\hra\c^2$ such that $\psi_{j-1}(\gamma)\subset r_j\b\setminus s_{j-1}'\overline\b$. Thus, in view of properties {\rm (1$_{j-1}$)}--{\rm (6$_{j-1}$)}, Mergelyan theorem with interpolation applied to $\psi_{j-1}\colon D_{j-1}\cup\gamma\hra\c^2$ guarantees the existence of a connected, smoothly bounded, compact domain $D_{j-1}'\subset \mathring M_j$ and an embedding $\psi_{j-1}'\colon D_{j-1}'\hra\c^2$ of class $\Ascr^1(D_{j-1}')$ with the following properties:
\begin{enumerate}[\rm (1$_{j-1}'$)]
\item $D_{j-1}\Subset D_{j-1}'\Subset M_j$ and $D_{j-1}'$ is a strong deformation retract of $M_j$.
\item $|\psi_{j-1}'(p)-\psi_{j-1}(p)|<\epsilon_j/2$ for all $p\in D_{j-1}$.
\item $\dist_{\psi_{j-1}'}(p_0,\sfb D_{j-1}')>j-1$.
\item $\psi_{j-1}'(\sfb D_{j-1}')\subset r_j\b\setminus s_{j-1}'\overline\b$.
\item $\psi_{j-1}'(D_{j-1}'\setminus \mathring D_{j-1})\cap s_{j-1}'\overline\b=\emptyset$.
\item $\Lambda_i\subset \psi_{j-1}'(\mathring D_i)$ for all $i\in\{0,\ldots,j-1\}$. 
\end{enumerate}
Since $\chi(M_j\setminus \mathring D_{j-1}')=0$, this reduces the proof to Case 1 and hence concludes the construction of the sequence $\Xi_j=\{D_j,\psi_j\}$ meeting properties {\rm (1$_j$)}--{\rm (6$_j$)}, $j\in\n$.
 
This completes the proof of Theorem \ref{th:main}.
\end{proof}

To finish the paper we show how Theorem \ref{th:main} implies Theorem \ref{th:main-intro}.
%
%
\begin{proof}[Proof of Theorem \ref{th:main-intro}]
Let $\Lambda$ and $M$ be as in Theorem \ref{th:main-intro} and assume without loss of generality that $0\in\Lambda$. By Alarc\'on and L\'opez \cite[Main Theorem, page 1795]{AlarconLopez2013JGA}, there exists a complex structure on $M$ such that the open Riemann surface $M$ admits a proper holomorphic embedding $\phi\colon M\hra\c^2$; up to composing with a translation we may assume that $0\in\phi(M)$. Take a number $0<r<1$ such that $\Lambda\cap r\overline\b=\{0\}$; since $\Lambda\subset\b$ is closed and discrete, every small enough $r>0$ meets this requirement. Thus, there are a small number $0<s<r$ and a smoothly bounded compact disk $K\subset M$ such that $\phi(\sfb K)\subset r\b\setminus s\overline\b$ and $0\in\phi(\mathring K)$. 

Now, Theorem \ref{th:main} applied to $\psi:=\phi|_K$ furnishes a domain $D\subset M$ homeomorphic to $M$ and a complete proper holomorphic embedding $\wt\psi\colon D\hra \b$ with $\Lambda\subset \wt\psi(D)$. Since $D$ is homeomorphic to the smooth surface $M$, there exists a complex structure on $M$ (possibly different from the one used in the previous paragraph) such that the open Riemann surface $M$ is biholomorphic to $D$. This concludes the proof.
\end{proof}


\subsection*{Acknowledgements}
We thank Francisco J.\ L\'opez for helpful discussions about the paper.

A.\ Alarc\'on is supported by the Ram\'on y Cajal program of the Spanish Ministry of Economy and Competitiveness and partially supported by the MINECO/FEDER grant no. MTM2014-52368-P, Spain.



\begin{thebibliography}{10}

\bibitem{AlarconDrinovecForstnericLopez2015PLMS}
A.~Alarc{\'o}n, B.~Drinovec~Drnov{\v{s}}ek, F.~Forstneri{\v{c}}, and F.~J.
  L{\'o}pez.
\newblock Every bordered {R}iemann surface is a complete conformal minimal
  surface bounded by {J}ordan curves.
\newblock {\em Proc. Lond. Math. Soc. (3)}, 111(4):851--886, 2015.

\bibitem{AlarconForstneric2013MA}
A.~Alarc{\'o}n and F.~Forstneri{\v{c}}.
\newblock Every bordered {R}iemann surface is a complete proper curve in a
  ball.
\newblock {\em Math. Ann.}, 357(3):1049--1070, 2013.

\bibitem{AlarconGlobevnikLopez2016Crelle}
A.~Alarc{\'o}n, J.~Globevnik, and F.~J. L{\'o}pez.
\newblock {A construction of complete complex hypersurfaces in the ball with
  control on the topology}.
\newblock {\em J. Reine Angew. Math.}, in press.

\bibitem{AlarconLopez2013MA}
A.~Alarc{\'o}n and F.~J. L{\'o}pez.
\newblock Null curves in $\mathbb{C}^3$ and {C}alabi-{Y}au conjectures.
\newblock {\em Math. Ann.}, 355(2):429--455, 2013.

\bibitem{AlarconLopez2013JGA}
A.~Alarc{\'o}n and F.~J. L{\'o}pez.
\newblock Proper holomorphic embeddings of {R}iemann surfaces with arbitrary
  topology into {$\mathbb{C}^2$}.
\newblock {\em J. Geom. Anal.}, 23(4):1794--1805, 2013.

\bibitem{AlarconLopez2016JEMS}
A.~Alarc{\'o}n and F.~J. L{\'o}pez.
\newblock Complete bounded embedded complex curves in {$\mathbb{C}^2$}.
\newblock {\em J. Eur. Math. Soc. (JEMS)}, 18(8):1675--1705, 2016.

\bibitem{BellNarasimhan1990book}
S.~R. Bell and R.~Narasimhan.
\newblock Proper holomorphic mappings of complex spaces.
\newblock In {\em Several complex variables, {VI}}, volume~69 of {\em
  Encyclopaedia Math. Sci.}, pages 1--38. Springer, Berlin, 1990.

\bibitem{Bishop1958PJM}
E.~Bishop.
\newblock Subalgebras of functions on a {R}iemann surface.
\newblock {\em Pacific J. Math.}, 8:29--50, 1958.

\bibitem{Drinovec2015JMAA}
B.~Drinovec~Drnov{\v{s}}ek.
\newblock Complete proper holomorphic embeddings of strictly pseudoconvex
  domains into balls.
\newblock {\em J. Math. Anal. Appl.}, 431(2):705--713, 2015.

\bibitem{Forstneric2011book}
F.~Forstneri{\v{c}}.
\newblock {\em Stein manifolds and holomorphic mappings}, volume~56 of {\em
  Ergebnisse der Mathematik und ihrer Grenzgebiete. 3. Folge. A Series of
  Modern Surveys in Mathematics [Results in Mathematics and Related Areas. 3rd
  Series. A Series of Modern Surveys in Mathematics]}.
\newblock Springer, Heidelberg, 2011.
\newblock The homotopy principle in complex analysis.

\bibitem{ForstnericGlobevnikStensones1996MA}
F.~Forstneri\v{c}, J.~Globevnik, and B.~Stens{\o}nes.
\newblock Embedding holomorphic discs through discrete sets.
\newblock {\em Math. Ann.}, 305(3):559--569, 1996.

\bibitem{ForstnericWold2009JMPA}
F.~Forstneri\v{c} and E.~F. Wold.
\newblock Bordered {R}iemann surfaces in {$\mathbb{C}^2$}.
\newblock {\em J. Math. Pures Appl. (9)}, 91(1):100--114, 2009.

\bibitem{ForstnericWold2013APDE}
F.~Forstneri\v{c} and E.~F. Wold.
\newblock Embeddings of infinitely connected planar domains into {${\mathbb
  C}^2$}.
\newblock {\em Anal. PDE}, 6(2):499--514, 2013.

\bibitem{Globevnik2015AM}
J.~Globevnik.
\newblock A complete complex hypersurface in the ball of {$\Bbb C^N$}.
\newblock {\em Ann. of Math. (2)}, 182(3):1067--1091, 2015.

\bibitem{Globevnik2016JMAA}
J.~Globevnik.
\newblock Embedding complete holomorphic discs through discrete sets.
\newblock {\em J. Math. Anal. Appl.}, 444(2):827--838, 2016.

\bibitem{Globevnik2016MA}
J.~Globevnik.
\newblock Holomorphic functions unbounded on curves of finite length.
\newblock {\em Math. Ann.}, 364(3-4):1343--1359, 2016.

\bibitem{Jones1979PAMS}
P.~W. Jones.
\newblock A complete bounded complex submanifold of {${\bf C}^{3}$}.
\newblock {\em Proc. Amer. Math. Soc.}, 76(2):305--306, 1979.

\bibitem{MartinUmeharaYamada2009PAMS}
F.~Mart\'in, M.~Umehara, and K.~Yamada.
\newblock Complete bounded holomorphic curves immersed in {$\mathbb{C}^2$} with
  arbitrary genus.
\newblock {\em Proc. Amer. Math. Soc.}, 137(10):3437--3450, 2009.

\bibitem{Mergelyan1951DAN}
S.~N. Mergelyan.
\newblock On the representation of functions by series of polynomials on closed
  sets.
\newblock {\em Doklady Akad. Nauk SSSR (N.S.)}, 78:405--408, 1951.

\bibitem{Runge1885AM}
C.~Runge.
\newblock Zur {T}heorie der {A}nalytischen {F}unctionen.
\newblock {\em Acta Math.}, 6(1):245--248, 1885.

\bibitem{Yang1977}
P.~Yang.
\newblock Curvature of complex submanifolds of {$C^{n}$}.
\newblock In {\em Several complex variables ({P}roc. {S}ympos. {P}ure {M}ath.,
  {V}ol. {XXX}, {P}art 2, {W}illiams {C}oll., {W}illiamstown, {M}ass., 1975)},
  pages 135--137. Amer. Math. Soc., Providence, R.I., 1977.

\bibitem{Yang1977JDG}
P.~Yang.
\newblock Curvatures of complex submanifolds of {${\bf C}^{n}$}.
\newblock {\em J. Differential Geom.}, 12(4):499--511 (1978), 1977.

\end{thebibliography}

\def\cprime{$'$}


\vspace*{0.3cm}
\noindent Antonio Alarc\'{o}n

\noindent Departamento de Geometr\'{\i}a y Topolog\'{\i}a e Instituto de Matem\'aticas (IEMath-GR), Universidad de Granada, Campus de Fuentenueva s/n, E--18071 Granada, Spain.

\noindent  e-mail: {\tt alarcon@ugr.es}

\vspace*{0.3cm}

\noindent Josip Globevnik

\noindent Department of Mathematics, University of Ljubljana, and Institute
of Mathematics, Physics and Mechanics, Jadranska 19, SI--1000 Ljubljana, Slovenia.

\noindent e-mail: {\tt josip.globevnik@fmf.uni-lj.si}

\end{document}